\documentclass[11pt]{amsart}

\usepackage{bfc}

\usepackage{hyperref}

\usepackage[margin=1.25in]{geometry}

\usepackage[utf8]{inputenc}

\usepackage{enumerate}

\begin{document}

\title{Isometries of optimal pseudo-Riemannian metrics}
\author{Brian Clarke}
\date{\today}
\thanks{This research was supported by NSF grant DMS-0902674.}
\address{Department of Mathematics, Stanford University, Stanford, CA
  94305-2125}
\email{bfclarke@stanford.edu}
\urladdr{http://math.stanford.edu/$\sim$bfclarke/}

\begin{abstract}
  We give a concise proof that large classes of optimal (constant
  curvature or Einstein) pseudo-Riemannian metrics are maximally
  symmetric within their conformal class.
\end{abstract}

\maketitle{}

The study of optimal metrics has held a central position in the study
of Riemannian and pseudo-Riemannian geometry.  In this note, we give a
very short alternative proof---for some classes of special cases---of
general results of Hebey--Vaugon \cite{hebey93:_le_yamab} and Ferrand
\cite{Ferrand:1996tb} on the symmetry of optimal Riemannian metrics.
Via the same short proof, we also extend these results to some
pseudo-Riemannian and non-compact cases.  In particular, we show that
many optimal (constant curvature or Einstein) pseudo-Riemannian
metrics are, in a very strong sense, the most symmetric metrics within
their conformal class.  The general result says that if this optimal metric is
(after an appropriate normalization) unique within its conformal
class, then any conformal transformation (and, in particular,
isometry) of a given pseudo-Riemannian metric in that class must be an
isometry of the optimal metric.  We then give some circumstances in
which this general result applies.

Throughout the paper, we assume all manifolds are connected and, for
simplicity, that they are smooth.  They may be open, closed, or with
boundary.  Also for simplicity, we assume all Riemannian metrics are
smooth, though they may be complete or incomplete unless otherwise
mentioned.

\section{General results}
\label{sec:general-statements}

We begin with the following definition.

\begin{definition}\label{dfn:1}
  Let $\mathcal{G}$ be a collection of pseudo-Riemannian metrics on
  a manifold $M$, and let $g_0 \in \mathcal{G}$.  We say $g_0$ is
  \emph{strongly maximally symmetric within $\mathcal{G}$} if, for all
  $g \in \mathcal{G}$ and for each isometry $\varphi \in
  \textnormal{Diff}(M)$ of $g$, $\varphi$ is an isometry of $g_0$.
\end{definition}

Let's denote the isometry group of a metric $g$ by $I_g$.  An
obviously equivalent formulation of the definition is the following.
A metric $g_0$ which is strongly maximally symmetric within
$\mathcal{G}$ has the property that, for any $g \in \mathcal{G}$, $I_g$
is a subgroup of $I_{g_0}$ when both are considered as subgroups of
$\textnormal{Diff}(M)$.

Our main theorem can now be stated in full generality as follows.
Consider the set $\mathcal{M}$ of all Riemannian metrics on $M$, and
let $\mathcal{N} \subseteq \mathcal{M}$ be any
$\textnormal{Diff}(M)$-invariant subset.  (For example, $\mathcal{N}$
could be all metrics with constant scalar curvature $-1$.)  For any
metric $g$ on $M$, we denote by $C_g$ the group of conformal
diffeomorphisms of $(M, g)$, that is, the set of $\varphi \in
\textnormal{Diff}(M)$ with $\varphi^* g = \sigma g$ for some $\sigma
\in C^\infty(M)$.

\begin{theorem}\label{thm:1}
  Let a conformal class $\mathcal{C}$ of pseudo-Riemannian metrics on
  a manifold $M$ be given.  If there is a unique metric $g_0 \in
  \mathcal{C} \cap \mathcal{N}$, then for each $g \in \mathcal{C}$,
  $C_g$ is a subgroup of $I_{g_0}$.  In particular, $g_0$ is strongly
  maximally symmetric within $\mathcal{C}$.
\end{theorem}
\begin{proof}
  For any function $\rho \in C^\infty(M)$, let $g := e^\rho g_0 \in
  \mathcal{C}$.  Let $\varphi \in C_g$ be any conformal diffeomorphism
  of $g$, say $\varphi^* g = e^\sigma g$.  Then we have
  \begin{equation}\label{eq:1}
    \varphi^* g_0 = \varphi^* (e^{-\rho} g) = \varphi^*
    (e^{-\rho})(\varphi^* g) = (\varphi^* e^\rho)^{-1} e^\sigma g.
  \end{equation}
  Therefore $\varphi^* g_0$ is conformal to $g_0$.  On the other hand,
  since $\mathcal{N}$ is diffeomorphism-invariant, $\varphi^* g_0 \in
  \mathcal{N}$.  But by assumption, $g_0$ is the only metric in
  $\mathcal{C} \cap \mathcal{N}$, so in fact $\varphi^* g_0 = g_0$,
  i.e., $\varphi \in I_{g_0}$.
\end{proof}

\begin{remark}\label{rmk:1}
  The reader should compare Theorem \ref{thm:1} (and its corollaries
  below) with Ferrand's proof of the Lichnerowicz conjecture
  \cite[Thm.~A]{Ferrand:1996tb} (which was proved by Obata
  \cite{obata72:_rieman} in the compact case for the connected
  component of the identity in $C_g$, and which can be proved via a
  simpler argument for $\dim M = 2$) stating that if $\dim M \geq 3$
  and a Riemannian manifold $(M, g)$ is not conformally equivalent to
  a round sphere or Euclidean space, then $C_g$ is \emph{inessential}.
  That is, it is contained in the isometry group of some metric $g_0$
  in the conformal class of $g$.  We note that these results do not
  explicitly identify the metric $g_0$.

  In particular, one should also compare Theorem \ref{thm:1} to the
  resolution of the equivariant Yamabe problem by Hebey--Vaugon
  \cite{hebey93:_le_yamab}, which uses a different strategy to show
  that for any compact Riemannian manifold $(M, g)$ ($\dim M \geq 3$),
  there is an $I_g$-invariant metric $g_0$ that is conformal to $g$
  and has constant scalar curvature.

  The primary differences between our result and above-mentioned
  results are threefold.  First, our proof is more concise.  Second,
  it is applicable to certain pseudo-Riemannian cases not handled in
  \cite{Ferrand:1996tb,hebey93:_le_yamab} (however, one should compare
  \cite{Frances:2010fm}), and certain noncompact cases not handled in
  \cite{hebey93:_le_yamab}.  Third, our proof is deficient in that it
  only solves the Lichnerowicz conjecture or the equivariant Yamabe
  problem in dimension two (cf.~\S \ref{sec:const-gauss-curv}; the
  latter is also solved in dimension two via our method only when
  assuming $M \neq S^2$), while only partially recovering these
  results in dimension three and higher (cf.~\S
  \ref{sec:const-sect-curv} and \S \ref{sec:einstein-metrics}).
\end{remark}

For any function $\zeta$ on $M$, let $I_\zeta = \{ \varphi \in
\textnormal{Diff}(M) \mid \varphi^* \zeta = \zeta \}$ denote its
isotropy group.

Working along the same lines as the proof of Theorem \ref{thm:1}, we
get the following characterization of the isometries of a metric in
the conformal class of a strongly maximally symmetric metric.

\begin{theorem}\label{thm:2}
  Let a conformal class $\mathcal{C}$ of pseudo-Riemannian metrics
  on a manifold $M$ be given.  If $g_0 \in \mathcal{C}$ is strongly
  maximally symmetric within $\mathcal{C}$ and $g = \rho g_0$ for some
  positive function $\rho$ on M, then $I_g = I_\rho \cap I_{g_0}$.
\end{theorem}
\begin{proof}
  If $\varphi \in I_g$, then by assumption $\varphi \in I_{g_0}$.
  Similarly to (\ref{eq:1}), we have
  \begin{equation*}
    \rho^{-1} g = g_0 = \varphi^* g_0 = (\varphi^* \rho)^{-1} g,
  \end{equation*}
  which implies that $\varphi^* \rho = \rho$.
\end{proof}

\section{Applications}
\label{sec:applications}

Let us now give some examples of subsets $\mathcal{N} \subseteq
\mathcal{M}$ for which there are conformal classes containing a unique
metric in $\mathcal{N}$, so that the theorems of the previous section
apply.  We will give the statements corresponding to Theorem
\ref{thm:1}; the statements for Theorem \ref{thm:2} are of course
analogous.  In this section, we assume $\partial M = \emptyset$.

\subsection{Constant Gaussian curvature}
\label{sec:const-gauss-curv}

Let $M$ be two-dimensional, and let $g$ be any Riemannian metric on
$M$.  By Poincaré uniformization, there exists a metric $g_0$ that is
conformal to $g$ and is complete with constant Gaussian curvature
$\eta$, where $\eta \in \{ +1, 0, -1 \}$.  If $M \cong S^2$ (i.e.,
$\eta = +1$), then $g_0$ is not unique.  If $\eta = 0$, then $g_0$ is
unique only up to homothety if it is isometric to the Euclidean plane.
On the other hand, $(M, g_0)$ is unique if it is isometric to a
cylinder or a torus if we additionally require $\inj(M, g_0) = 1$.  If
$\eta = -1$, then $g_0$ is unique.  If $ \eta = -1$ and $M$ is compact
(i.e., a surface of genus $p \geq 2$), then Hurwitz's Theorem
\cite[p.~258]{farkas92:_rieman_surfac} says that the order of
$I_{g_0}$ is finite and bounded above by $84(p - 1)$.


Thus, we may in this case let $\mathcal{N}$ be the subset of metrics
$g_0$ with constant Gaussian curvature.  If the curvature of $g_0$
is $0$ and $g_0$ is not the Euclidean metric on the plane, then we
additionally require $\inj(M, g_0) = 1$.  We obtain the following
corollary of Theorem \ref{thm:1}.

\begin{corollary}\label{cor:2}
  Let $(M, g)$ be any smooth Riemannian 2-manifold, and let $g_0$ be
  any metric with constant Gaussian curvature conformal to $g$.  If
  $(M, g_0)$ is not isometric to the sphere with its round metric or
  the Euclidean plane, then any conformal diffeomorphism of $g$ is an
  isometry of $g_0$.

  Furthermore, if $M$ is compact with genus $p \geq 2$, then the
  orders of $C_g$ and $I_g$ are finite and bounded above by $84(p -
  1)$.  In particular, $(M, g)$ admits no nontrivial conformal Killing
  fields.
\end{corollary}

\begin{remark}\label{rmk:2}
  The previous corollary also follows from a result of Calabi
  \cite[Thm.~3]{Calabi:1985ua} (cf.~also
  \cite{lichnerowicz57:_sur,Matsushima:1957uw}).  One can also deduce
  the previous corollary using the fact that isometries of the
  constant-curvature metric on a surface (other than $S^2$ and $\C$)
  agree with biholomorphisms of the complex structure.  Again, the
  advantage of our proof is its simplicity.
\end{remark}

\subsection{Constant scalar curvature}
\label{sec:const-sect-curv}

For the rest of the paper, we consider manifolds $M$ with $\dim M \geq
3$.  If $(M, g)$ is a compact Riemannian manifold, then by
Trudinger, Aubin, and Schoen's resolution of the Yamabe problem
\cite{Trudinger:1968vh,Aubin:1976ul,Schoen:1984wk}, there exists a
metric $g_0$ conformal to $g$ of constant scalar curvature.

In general, the metric $g_0$ need not be the unique (even up to
homothety) metric conformal to $g$ of constant scalar curvature, so
that Theorem \ref{thm:1} does not necessarily apply.  However, if the
scalar curvature of $g_0$ is nonpositive, then a maximum principle
argument implies that $g_0$ is, in fact, unique up to homothety
\cite[Thm.~4.2]{Kazdan:1975up}.  Therefore, letting $\mathcal{N}$ be
the subset of metrics on $M$ with unit volume and constant scalar
curvature, we obtain:

\begin{corollary}\label{cor:3}
  Let $(M, g)$ be a smooth, compact Riemannian manifold with
  nonpositive Yamabe invariant.  Then any conformal diffeomorphism of
  $g$ is an isometry of the unique metric $g_0$ conformal to $g$ with
  constant scalar curvature and unit volume.
\end{corollary}

\subsection{Einstein metrics}
\label{sec:einstein-metrics}

We now allow the case where $(M, g)$ is pseudo-Riemannian, and not
just Riemannian.  Using a result of Kühnel and Rademacher
\cite[Thm.~$1^*$]{Kuhnel:1995up}, we can show the following.  We
assume that the number of negative eigenvalues in the signature of $g$
is no greater than $\frac{n}{2}$.  The results of this subsection apply
even if $g$ is only of regularity $C^3$.

If there is an Einstein metric $g_0$ in the conformal class of $g$,
then this metric is unique up to homothety, except in the 
following two cases:
\begin{enumerate}[\ \ \ \ \ (1)]
\item $(M, g_0)$ is a simply-connected Riemannian space of constant
  sectional curvature.
\item $(M, g_0)$ is a warped-product manifold $\R \times_{e^{2t}} N$,
  where $N$ is a complete Ricci-flat $(n-1)$-dimensional Riemannian
  manifold ($n = \dim M$).  Explicitly, the metric is given by $\pm 
  d t^2 + e^{2t} h$, where $h$ is the metric on $N$ and the sign of $d 
  t^2$ depends on the signature of $g_0$.
\end{enumerate}
If neither of the above cases holds, we can choose the Einstein metric
$g_0$ uniquely (i.e., fix a scale) in some circumstances, e.g., the
following:
\begin{enumerate}[\ \ \ \ \ (a)]
\item If the scalar curvature of $g_0$ is positive or negative, scale
  so that it is $\pm 1$.
\item If $\Vol(M, g_0) < \infty$, scale to unit volume.
\item If $0 < \inj(M, g_0) < \infty$, scale so that $\inj(M, g_0) = 1$.
\end{enumerate}
Thus, letting $\mathcal{N}$ be the set of such metrics, we have the
following:

\begin{corollary}\label{cor:4}
  Let $(M, g)$ be a pseudo-Riemannian manifold of regularity at least
  $C^3$, with $\dim M \geq 3$.  Assume that $g$ has no greater than
  $\frac{n}{2}$ negative eigenvalues.  If $(M, g)$ is conformally
  equivalent to a (geodesically) complete Einstein manifold $(M,
  g_0)$, for which neither case (1) nor (2) above holds, and for which
  either (a), (b), or (c) applies, then every conformal diffeomorphism
  of $g$ is an isometry of $g_0$.
\end{corollary}

\subsection*{Acknowledgements}
\label{sec:acknowledgements}

I would like to thank to Jürgen Jost for discussing the case of
Riemann surfaces and Jacob Bernstein for suggesting that the argument
of Theorem \ref{thm:1} might be more generally applicable.  I would
also like to thank Yanir Rubinstein for pointing out
\cite{Calabi:1985ua,lichnerowicz57:_sur,Matsushima:1957uw}, as well as Karin
Melnick for several helpful comments.

\bibliography{papers,main}

\providecommand{\bysame}{\leavevmode\hbox to3em{\hrulefill}\thinspace}
\providecommand{\MR}{\relax\ifhmode\unskip\space\fi MR }
\providecommand{\MRhref}[2]{%
  \href{http://www.ams.org/mathscinet-getitem?mr=#1}{#2}
}
\providecommand{\href}[2]{#2}
\begin{thebibliography}{Aub76}

\bibitem[Aub76]{Aubin:1976ul}
Thierry Aubin, \emph{{Equations diff{\'e}rentielles non lin{\'e}aires et
  probl{\`e}me de Yamabe concernant la courbure scalaire}}, J. Math. Pures
  Appl. (9) \textbf{55} (1976), no.~3, 269--296.

\bibitem[Cal85]{Calabi:1985ua}
Eugenio Calabi, \emph{{Extremal K{\"a}hler metrics. II}}, Differential geometry
  and complex analysis, Springer, Berlin, 1985, pp.~95--114.

\bibitem[Fer96]{Ferrand:1996tb}
Jacqueline Ferrand, \emph{{The action of conformal transformations on a
  Riemannian manifold}}, Math. Ann. \textbf{304} (1996), no.~2, 277--291.

\bibitem[FK92]{farkas92:_rieman_surfac}
Hershel~M. Farkas and Irwin Kra, \emph{Riemann surfaces}, 2nd ed., Graduate
  Texts in Mathematics, vol.~71, Springer, New York, 1992.

\bibitem[FM10]{Frances:2010fm}
Charles Frances and Karin Melnick, \emph{{Conformal actions of nilpotent groups
  on pseudo-Riemannian manifolds}}, Duke Math. J. \textbf{153} (2010), no.~3,
  511--550.

\bibitem[HV93]{hebey93:_le_yamab}
Emmanuel Hebey and Michel Vaugon, \emph{Le probl\`eme de {Y}amabe
  \'equivariant}, Bull. Sci. Math. \textbf{117} (1993), no.~2, 241--286.
  \MR{1216009 (94k:53056)}

\bibitem[KR95]{Kuhnel:1995up}
Wolfgang K{\"u}hnel and Hans-Bert Rademacher, \emph{{Conformal diffeomorphisms
  preserving the Ricci tensor}}, Proc. Amer. Math. Soc. \textbf{123} (1995),
  no.~9, 2841--2848.

\bibitem[KW75]{Kazdan:1975up}
Jerry~L. Kazdan and F.~W. Warner, \emph{{Scalar curvature and conformal
  deformation of Riemannian structure}}, J. Differential Geom. \textbf{10}
  (1975), 113--134.

\bibitem[Lic57]{lichnerowicz57:_sur}
Andr{\'e} Lichnerowicz, \emph{Sur les transformations analytiques des
  vari\'et\'es k\"ahl\'eriennes compactes}, C. R. Acad. Sci. Paris \textbf{244}
  (1957), 3011--3013. \MR{0094479 (20 \#996)}

\bibitem[Mat57]{Matsushima:1957uw}
Yoz{\^o} Matsushima, \emph{{Sur la structure du groupe d'hom\'eomorphismes
  analytiques d'une certaine vari\'et\'e k\"ahl\'erienne}}, Nagoya Math. J.
  \textbf{11} (1957), 145--150.

\bibitem[Oba72]{obata72:_rieman}
Morio Obata, \emph{The conjectures on conformal transformations of {R}iemannian
  manifolds}, J. Differential Geometry \textbf{6} (1971/72), 247--258.
  \MR{0303464 (46 \#2601)}

\bibitem[Sch84]{Schoen:1984wk}
Richard Schoen, \emph{{Conformal deformation of a Riemannian metric to constant
  scalar curvature}}, J. Differential Geom. \textbf{20} (1984), no.~2,
  479--495.

\bibitem[Tru68]{Trudinger:1968vh}
Neil~S. Trudinger, \emph{{Remarks concerning the conformal deformation of
  Riemannian structures on compact manifolds}}, Ann. Scuola Norm. Sup. Pisa (3)
  \textbf{22} (1968), 265--274.

\end{thebibliography}
\bibliographystyle{amsalpha}

\end{document}